\documentclass[3p]{elsarticle}
\newif \ifDdP \DdPtrue

\ifDdP
\else
\newcommand{\bigtimes}{\times}

\fi
\usepackage[T1]{fontenc}
\usepackage{amsmath,amsfonts,amsthm}
\usepackage[english,francais]{babel}
\usepackage{enumerate}
\usepackage{general}
\usepackage{hyperref}
\usepackage[latin1]{inputenc}
\usepackage{nicefrac}
\usepackage{numprint} 
\usepackage{paralist}
\ifDdP
\usepackage{pgfplots,pgfplotstable}
\usepackage{subcaption}
\fi
\usepackage{tikz}
\usepackage{xcolor}

\usepackage[bold]{hhtensor}

\ifDdP
\usetikzlibrary{external}
\pgfqkeys{/pgfplots}{
  cycle list name = black white,
}
\fi

\newcommand{\lproj}[2][h]{\pi_{#1}^{#2}}

\newcommand{\UT}[1][k]{\underline{\vec{\mathsf{U}}}_T^{#1}}
\newcommand{\Uh}[1][k]{\underline{\vec{\mathsf{U}}}_h^{#1}}
\newcommand{\UhD}[1][k]{\underline{\vec{\mathsf{U}}}_{h,0}^{#1}}

\newcommand{\DT}[1][k]{D_T^{#1}}
\newcommand{\ST}[1][k]{\matr{S}_T^{#1}}
\newcommand{\pT}[1][k]{\vec{p}_T^{#1}}
\newcommand{\PT}[1][k]{\vec{P}_T^{#1}}

\newcommand{\IT}[1][k]{\underline{\vec{I}}_T^{#1}}

\newcommand{\vu}{\vec{u}}

\newcommand{\vv}{\vec{v}}
\newcommand{\vw}{\vec{w}}

\newcommand{\vf}[1][]{\vec{f}_{#1}}

\newcommand{\su}{\vec{\mathsf{u}}}
\newcommand{\hsu}{\vec{\widehat{\mathsf{u}}}}
\newcommand{\sv}{\vec{\mathsf{v}}}
\newcommand{\sw}{\vec{\mathsf{w}}}

\newcommand{\usu}[1][T]{\underline{\vec{\mathsf{u}}}_{#1}}
\newcommand{\husu}[1][T]{\underline{\vec{\widehat{\mathsf{u}}}}_{#1}}
\newcommand{\tusw}[1][T]{\underline{\vec{\widetilde{\mathsf{w}}}}_{#1}}
\newcommand{\usv}[1][T]{\underline{\vec{\mathsf{v}}}_{#1}}
\newcommand{\usw}[1][T]{\underline{\vec{\mathsf{w}}}_{#1}}

\newcommand{\cT}{\underline{\vec{\mathsf{c}}}_T^k}

\newcommand{\flux}[1][TF]{\vec{\tau}_{#1}}

\journal{Comptes Rendus de l'Académie des sciences Série I}

\ifDdP
\pgfplotscreateplotcyclelist{list1}{
{every mark/.append style={fill=gray},mark=*},
{every mark/.append style={solid,fill=gray},mark=square*},
{every mark/.append style={solid,fill=gray},mark=otimes*},
{mark=o,every mark/.append style={solid,fill=gray}}}

\pgfplotscreateplotcyclelist{list2}{%
{densely dashed,every mark/.append style={solid,fill=gray},mark=*},
{densely dashed,every mark/.append style={solid,fill=gray},mark=square*},
{densely dashed,every mark/.append style={solid,fill=gray},mark=otimes*},
{densely dashed,mark=o,every mark/.append style={solid,fill=gray}}}
\fi

\begin{document}

\begin{frontmatter}

\selectlanguage{english}

\title{Equilibrated tractions for the Hybrid High-Order method}
\author[1]{Daniele A. Di Pietro}
\ead{daniele.di-pietro@univ-montp2.fr}
\address[1]{%
  University of Montpellier 2, I3M, 34057 Montpellier CEDEX 5, France%
}
\author[2]{Alexandre Ern}
\ead{ern@cermics.enpc.fr}
\address[2]{%
  University Paris-Est,
  CERMICS (ENPC), 6--8 avenue Blaise Pascal, 77455,
  Marne la Vall\'ee cedex 2, France%
}

\begin{abstract}
  We show how to recover equilibrated face tractions for the hybrid high-order method for linear elasticity recently introduced in~\cite{Di-Pietro.Ern:15}, and prove that these tractions are optimally convergent.
\vskip 0.5\baselineskip

\selectlanguage{francais}
\noindent{\bf R\'esum\'e} \vskip 0.5\baselineskip \noindent {\bf Tractions équilibrées pour la méthode hybride d'ordre élevé.}
Nous montrons comment obtenir des tractions de face équilibrées pour la méthode hybride d'ordre élevé pour l'élasticité linéaire rrécemment introduite dans~\cite{Di-Pietro.Ern:15} et prouvons que ces tractions convergent de manière optimale.
\end{abstract}
\end{frontmatter}


\selectlanguage{english}

\section{Introduction}

Let $\Omega\subset\Real^d$, $d\in\{2,3\}$, denote a bounded connected polygonal or polyhedral domain. 
For $X\subset\overline\Omega$, we denote by $(\cdot,\cdot)_X$ and $\norm[X]{{\cdot}}$ respectively the standard inner product and norm of $L^2(X)$, and a similar notation is used for $L^2(X)^d$ and $L^2(X)^{d\times d}$.
For a given external load $\vf\in L^2(\Omega)^d$, we consider the linear elasticity  problem:
Find $\vu\in H_0^1(\Omega)^d$ such that
\begin{equation}
  \label{eq:weak}
  2\mu(\GRADs\vu,\GRADs\vv)_{\Omega} + \lambda(\DIV\vu,\DIV\vv)_{\Omega}
  = (\vf,\vv)_{\Omega}.
\end{equation}
with $\mu>0$ and $\lambda\ge 0$ real numbers representing the scalar Lam\'{e} coefficients and $\GRADs$ denoting the symmetric gradient operator.
Classically, the solution to~\eqref{eq:weak} satisfies $-\DIV\matr{\sigma}(\vu)=\vf$ a.e. in $\Omega$ with stress tensor $\matr{\sigma}(\vu)\eqbydef2\mu\GRADs\vu + \lambda\Id(\DIV\vu)$. Denoting by $T$ an open subset of $\Omega$ with non-zero Hausdorff measure ($T$ will represent a mesh element in what follows), partial integration yields the following local equilibrium property:
\begin{equation}
\label{eq:local.eq.cont}
(\matr{\sigma}(\vu),\GRADs\sv_T)_T 
- (\matr{\sigma}(\vu)\normal_T,\sv_T)_{\partial T}
= (\vf,\sv_T)_T
\qquad\forall\sv_T\in\Poly{k}(T)^d,
\end{equation}
where $\partial T$ and $\normal_T$ denote, respectively, the boundary and outward normal to $T$.
Additionally, the normal interface tractions $\matr{\sigma}(\vu)\normal_T$ are equilibrated across $\partial T\cap\Omega$.
The goal of this work is to 
\begin{inparaenum}[(i)]
\item devise a reformulation of the Hybrid High-Order method for linear elasticity introduced in~\cite{Di-Pietro.Ern:15} that identifies its local equilibrium properties expressed by a discrete counterpart of~\eqref{eq:local.eq.cont} and 
\item to show how the corresponding equilibrated face tractions can be obtained by element-wise post-processing.
\end{inparaenum}
This is an important complement to the original analysis, as local equilibrium is an essential property in practice.
The material is organized as follows: in Section~\ref{sec:hho} we outline the original formulation of the HHO method; in Section~\ref{sec:equilibrium} we derive the local equilibrium formulation based on a new local displacement reconstruction.


\section{The Hybrid High-Order method}
\label{sec:hho}

We consider admissible mesh sequences in the sense of~\cite[Section~1.4]{Di-Pietro.Ern:12}. Each mesh $\Th$ in the sequence is a finite collection $\{T\}$ of nonempty, disjoint, open, polytopic elements such that $\closure{\Omega}=\bigcup_{T\in\Th}\closure{T}$ and $h=\max_{T\in\Th}h_T$ (with $h_T$ the diameter of $T$), and there is a matching simplicial submesh of $\Th$ with locally equivalent mesh size and which is shape-regular in the usual sense. For all $T\in\Th$, the faces of $T$ are collected in the set $\Fh[T]$ and, for all $F\in\Fh[T]$, $\normal_{TF}$ is the unit normal to $F$ pointing out of $T$.
Additionally, interfaces are collected in the set $\Fhi$ and boundary faces in $\Fhb$.
The diameter of a face $F\in\Fh$ is denoted by $h_F$.
For the sake of brevity, we abbreviate $a\lesssim b$ the inequality $a\le Cb$ for positive real numbers $a$ and $b$ and a generic constant $C$ which can depend on the mesh regularity, on $\mu$, $d$, and the polynomial degree, but is independent of $h$ and $\lambda$.
We also introduce the notation $a\simeq b$ for the uniform equivalence $a \lesssim b\lesssim a$.


Let a polynomial degree $k\ge 1$ be fixed.
The local and global spaces of degrees of freedom (DOFs) are
\begin{equation}
  \label{eq:UT}
  \UT\eqbydef\Poly{k}(T)^d\times\left\{
  \bigtimes_{F\in\Fh[T]}\Poly[d-1]{k}(F)^d
  \right\}\quad\forall T\in\Th,
  \qquad
   \Uh\eqbydef\left\{
  \bigtimes_{T\in\Th}\Poly{k}(T)^d
  \right\}\times\left\{
  \bigtimes_{F\in\Fh}\Poly[d-1]{k}(F)^d
  \right\}.
\end{equation}
A generic collection of DOFs from $\Uh$ is denoted by $\usv[h]=\big((\sv_T)_{T\in\Th},(\sv_F)_{F\in\Fh}\big)$ and, for a given $T\in\Th$, $\usv=\big(\sv_T,(\sv_F)_{F\in\Fh[T]}\big)\in\UT$ indicates its restriction to $\UT$.
For all $T\in\Th$, we define a high-order local displacement reconstruction operator $\pT:\UT\to\Poly{k+1}(T)^d$ by solving the following (well-posed) pure traction problem:
For a given $\usv\in\UT$, $\pT\usv$ is such that 
\begin{equation}
  \label{eq:pT}
  (\GRADs\pT\usv,\GRADs\vw)_T 
  = (\GRADs\sv_T,\GRADs\vw)_T
  + \sum_{F\in\Fh[T]}(\sv_F-\sv_T,\GRADs\vw\,\normal_{TF})_F \qquad\forall \vw\in\Poly{k+1}(T)^d,
\end{equation}
and the rigid-body motion components of $\pT\usv$ are prescribed so that 
$\int_T\pT\usv=\int_T\sv_T$ and
$\int_T\GRADss(\pT\usv)= \sum_{F\in\Fh[T]} \int_F \frac12(\normal_{TF}{\otimes}\sv_F-\sv_F{\otimes}\normal_{TF})$ where $\GRADss$ is the skew-symmetric gradient operator. 
Additionally, we define the divergence reconstruction $\DT:\UT\to\Poly{k}(T)$ such that, for a given $\usv\in\UT$,
\begin{equation}
  \label{eq:DT}
  (\DT\usv,q)_T
  = (\DIV\sv_T,q)_T + \sum_{F\in\Fh[T]}(\sv_F-\sv_T,q\normal_{TF})_F\qquad \forall q\in\Poly{k}(T).
\end{equation}
We introduce the local bilinear form $a_T:\UT\times\UT\to\Real$ such that
\begin{equation}
  \label{eq:aT}
  a_T(\usw,\usv)\eqbydef 
  2\mu\left\{
  (\GRADs\pT\usw,\GRADs\pT\usv)_T
  + s_T(\usw,\usv)
  \right\}
  + \lambda(\DT\usw,\DT\usv)_T,
\end{equation}
where the stabilizing bilinear form $s_T:\UT\times\UT\to\Real$ is such that
\begin{equation}
  \label{eq:sT}
  s_T(\usw,\usv) \eqbydef 
  \sum_{F\in\Fh[T]} h_F^{-1} (\lproj[F]{k}(\PT\usw-\sw_F), \lproj[F]{k}(\PT\usv-\sv_F))_F,
\end{equation}
and a second displacement reconstruction $\PT:\UT\to\Poly{k+1}(T)$ is defined such that, for all $\usv\in\UT$,
$
\PT\usv\eqbydef \sv_T + (\pT\usv - \lproj[T]{k}\pT).
$
Let $\IT:H^{1}(T)^d\to\UT$ be the reduction map such that, for all $T\in\Th$ and all $\vv\in H^1(T)^d$,
$\IT\vv= \big(\lproj[T]{k}\vv,(\lproj[F]{k}\vv)_{F\in\Fh[T]}\big)$.
The potential reconstruction $\pT$ and the bilinear form $s_T$ are conceived so that they satisfy the following two key properties:
\begin{compactenum}[(i)]
\item \emph{Stability.} For all $\usv\in\UT$,
  \begin{equation}
  \label{eq:stability}
  \norm[T]{\GRADs\pT\usv}^2 + s_T(\usv,\usv)
  \simeq
  \norm[T]{\GRADs\sv_T}^2 + j_T(\usv,\usv),
  \end{equation} 
  with bilinear form $j_T:\UT\times\UT\to\Real$ such that
$j_T(\usw,\usv) \eqbydef \sum_{F\in\Fh[T]} h_F^{-1} (\sw_T-\sw_F, \sv_T-\sv_F)_F.$
\item \emph{Approximation.} For all $\vv\in H^{k+2}(T)^d$, 
\begin{equation}
  \label{eq:approx.pT.sT}
  \big\{
  \norm[T]{\GRADs(\vv-\pT\IT\vv)}^2 + s_T(\IT\vv,\IT\vv)
  \big\}^{\nicefrac12}\lesssim h_T^{k+1}\norm[H^{k+2}(T)^d]{\vv}.
\end{equation}
\end{compactenum}
We observe that, unlike $s_T$, the stabilization bilinear form $j_T$ only satisfies $j_T(\IT\vv,\IT\vv)\lesssim h^k\norm[H^{k+1}(T)^d]{\vv}$.
The discrete problem reads: Find $\usu[h]\in\UhD\eqbydef\{\usu[h]\in\Uh\st \su_F\equiv \vec{0}\quad\forall F\in\Fhb\}$ such that
\begin{equation}
  \label{eq:hho}
  a_h(\usu[h],\usv[h]) 
  \eqbydef\sum_{T\in\Th}a_T(\usw,\usv)
  = \sum_{T\in\Th}(\vf,\sv_T)_T\qquad\forall\usv[h]\in\UhD.
\end{equation}
The following convergence result was proved in~\cite{Di-Pietro.Ern:15}:
\begin{theorem}[Energy error estimate]
  \label{thm:conv.rate}
  Let $\vu\in H_0^1(\Omega)^d$ and $\usu[h]\in\UhD$ denote the unique solutions to~\eqref{eq:weak} and~\eqref{eq:hho}, respectively, and assume $\vu\in H^{k+2}(\Omega)^d$ and $\DIV\vu\in H^{k+1}(\Omega)$.
  Then, letting $\husu[h]\in\UhD$ be such that $\husu\eqbydef\IT\vu$ for all $T\in\Th$, the following holds (with $\norm[a,T]{\usv}^2=a_T(\usv,\usv)$ for all $\usv\in\UT$):
  \begin{equation}
    \label{eq:conv.rate.en}
    \sum_{T\in\Th}\norm[a,T]{\usu-\husu}^2
    \lesssim h^{2(k+1)}\left(
    \norm[H^{k+2}(\Omega)^d]{\vu}
    + \lambda\norm[H^{k+1}(\Omega)]{\DIV\vu}
    \right)^2.
  \end{equation}
Moreover, assuming elliptic regularity, 
$\sum_{T\in\Th}\norm[L^2(T)^d]{\vu-\pT\usu}^2
    \lesssim h^{2(k+2)}\left(
    \norm[H^{k+2}(\Omega)^d]{\vu}
    + \lambda\norm[H^{k+1}(\Omega)]{\DIV\vu}
    \right)^2$.
\end{theorem}


\section{Local equilibrium formulation}
\label{sec:equilibrium}

The difficulty in devising an equivalent local equilibrium formulation for problem~\eqref{eq:hho} comes from the stabilization term $s_T$, which introduces a non-trivial coupling of interface DOFs inside each element.
In this section, we introduce post-processed discrete displacement and stress reconstructions that allow us to circumvent this difficulty.
For a given element $T\in\Th$, define the following bilinear form on $\UT$:
\begin{equation}
  \label{eq:inner.prod.UT}
  \widetilde{a}_T(\usw,\usv) \eqbydef 
  2\mu\left\{(\GRADs\pT\usw,\GRADs\pT\usv)_T + j_T(\usw,\usv)\right\}
  + \lambda(\DT\usw,\DT\usv)_T,
\end{equation}
where the only difference with respect to the bilinear form $a_T$ defined by~\eqref{eq:aT} is that we have stabilized using $j_T$ instead of $s_T$.
We observe that, while proving a discrete local equilibrium relation for the method based on $\widetilde{a}_T$ would not require any local post-processing, the suboptimal consistency properties of $j_T$ would only yield $h^{2k}$ in the right-hand side of~\eqref{eq:conv.rate.en}.
Denoting by $\norm[\widetilde{a},T]{{\cdot}}$ the local seminorm induced by $\widetilde{a}_T$ on $\UT$, one can prove that, for all $\usv\in\UT$,
\begin{equation}
  \label{eq:local.norm.equiv:bis}
  \norm[\widetilde{a},T]{\usv}\simeq\norm[a,T]{\usv}.
\end{equation}
We next define the isomorphism $\cT:\UT\to\UT$ such that
\begin{equation}
  \label{eq:cT}
        \widetilde{a}_T(\cT\usw,\usv) = a_T(\usw,\usv) + (2\mu)j_T(\usw,\usv)\qquad\forall\usv\in\UT,
\end{equation}
and rigid-body motion components prescribed as above.
We also introduce the stress reconstruction $\ST:\UT\to\Poly{k}(T)^{d\times d}$ such that 
\begin{equation}
\label{eq:ST}
\ST\eqbydef (2\mu\GRADs\pT + \lambda\Id\DT) \circ\cT.
\end{equation}

\begin{lemma}[Equilibrium formulation]
The bilinear form $a_T$ defined by~\eqref{eq:aT} is such that, for all $\usw,\usv\in\UT$,
\begin{equation}
  \label{eq:aT.bis}
  a_T(\usw,\usv)
  = (\ST\usw,\GRADs\sv_T)_T
  + \sum_{F\in\Fh[T]}(\flux(\usw),\sv_F-\sv_T)_F,
\end{equation}
with interface traction $\flux:\UT\to\Poly[d-1]{k}(F)^d$ such that
\begin{equation}
  \label{eq:phiTF}
  \flux(\usw) = 
  \ST\usw\,\normal_{TF} + h_F^{-1}\left[
    \big((\cT\usw)_F-\sw_F\big) - \big((\cT\usw)_T-\sw_T\big)
    \right].
\end{equation}
\end{lemma}

\begin{proof}
Let $\tusw\eqbydef\cT\usw$.
We have, using the definitions~\eqref{eq:cT} of $\cT$ and~\eqref{eq:inner.prod.UT} of the bilinear form $\widetilde{a}_T$,
$$
\begin{alignedat}{2}
a_T(\usw,\usv) 
&= \widetilde{a}_T(\tusw,\usv) - (2\mu)j_T(\usw,\usv)
\\
&= 2\mu\left\{
(\GRADs\pT\tusw,\GRADs\pT\usv)_T 
+ j_T(\tusw-\usw,\usv)
\right\}
+ \lambda (\DT\tusw,\DT\usv)_T
\\
&= (\ST\usw,\GRADs\sv_T)_T
+ \sum_{F\in\Fh[T]} (\ST\usw\,\normal_{TF},\sv_F - \sv_T)_F
+  (2\mu) j_T(\tusw-\usw,\usv),
\end{alignedat}
$$
where we have  concluded using~\eqref{eq:pT} with $\vw=\pT\tusw$,~\eqref{eq:DT} with $q=\DT\tusw$, and recalling the definition~\eqref{eq:ST} of $\ST$.
To obtain~\eqref{eq:aT.bis}, it suffices to use the definition of $j_T$.
\end{proof}

\begin{lemma}[Local equilibrium]
  \label{lem:cons.hho}
  Let $\usu[h]\in\UhD$ denote the unique solution to~\eqref{eq:hho}. Then, for all $T\in\Th$, the following discrete counterpart of the local equilibrium relation~\eqref{eq:local.eq.cont} holds:
  \begin{equation}
  \label{eq:local.eq.disc}
    (\ST\usu,\GRADs\sv_T)_T
    - \sum_{F\in\Fh[T]}(\flux(\usu),\sv_T)_F
    = (\vf,\sv_T)_T
   \qquad\forall \sv_T\in\Poly{k}(T)^d,
  \end{equation}
  and the numerical flux are equilibrated in the following sense: 
  For all $F\in\Fhi$ such that $F\subset\partial T_1\cap\partial T_2$, 
  \begin{equation}
    \label{eq:cons.hho}
    \flux[T_1F](\usu[T_1]) + \flux[T_1F](\usu[T_2]) = \vec{0}.
  \end{equation}
\end{lemma}%

\begin{proof}
  To prove~\eqref{eq:local.eq.disc}, let an element $T\in\Th$ be fixed, take as an ansatz collection of DOFs in~\eqref{eq:hho} $\usv[h]=\big( (\sv_T)_{T\in\Th}, (\vec{0})_{F\in\Fh}\big)$ with $\sv_T$ in $\Poly{k}(T)^d$ and $\sv_{T'}\equiv\vec{0}$ for all $T'\in\Th\setminus\{T\}$, and use~\eqref{eq:aT.bis} with $\usw=\usu$ to conclude that $a_T(\usu,\usv)$ corresponds to the left-hand side of~\eqref{eq:local.eq.disc}.
  Similarly, to prove~\eqref{eq:cons.hho}, let an interface $F\in\Fhi$ be fixed and take as an ansatz collection of DOFs in~\eqref{eq:hho} $\usv[h]=\big((\vec{0})_{T\in\Th},(\sv_F)_{F\in\Fh}\big)\in\UhD$ with $\sv_F$ in $\Poly[d-1]{k}(F)^d$ and $\sv_{F'}\equiv\vec{0}$ for all $F'\in\Fh\setminus\{F\}$.
  Then, using~\eqref{eq:aT.bis} with $\usw=\usu$ in~\eqref{eq:hho}, it is inferred that
  $
  a_h(\usu[h],\usv[h]) 
  = (\flux[T_1F](\usu[T_1]) + \flux[T_2,F](\usu[T_2]),\sv_F)_F
  = 0,
  $
  which proves the desired result since $\flux[T_1F](\usu[T_1]) + \flux[T_2F](\usu[T_2])\in\Poly[d-1]{k}(F)^d$.
\end{proof}

To conclude, we show that the locally post-processed solution  yields a new collection of DOFs that is an equally good approximation of the exact solution as is the discrete solution $\usu[h]$.
Consequently, the equilibrated face numerical tractions defined in~\eqref{eq:phiTF} optimally converge to the exact tractions.

\begin{proposition}[{Convergence for $\cT\usu$}]
  \label{prop:conv.rate.cT}
  Using the notation of Theorem~\ref{thm:conv.rate}, the following holds:
  \begin{equation}
    \label{eq:conv.rate.cT}
    \sum_{T\in\Th}\norm[a,T]{\cT\usu-\husu}^2
    \lesssim h^{2(k+1)}\left(
    \norm[H^{k+2}(\Omega)^d]{\vu}
    + \lambda\norm[H^{k+1}(\Omega)]{\DIV\vu}
    \right)^2.
  \end{equation}
\end{proposition}

\begin{proof}
Let $T\in\Th$. Recalling~\eqref{eq:cT}, we have
$$
\begin{aligned}
  \widetilde{a}_T(\cT\usu-\husu,\usv)
  &=a_T(\usu,\usv)+(2\mu)j_T(\usu,\usv)-\widetilde{a}_T(\husu,\usv)
  \\
  &= a_T(\usu-\husu,\usv)
  + (2\mu) s_T(\husu,\usv)
  + (2\mu) j_T(\usu-\husu,\usv).
\end{aligned}
$$
Hence, using the Cauchy--Schwarz inequality followed by the stability property~\eqref{eq:stability} and multiple applications of the norm equivalence~\eqref{eq:local.norm.equiv:bis},
$$
\begin{aligned}
|\widetilde{a}_T(\cT\usu-\husu,\usv)|
&\le\left\{
\norm[a,T]{\usu-\husu}^2
+ (2\mu)s_T(\husu,\husu)
+ (2\mu)j_T(\usu-\husu,\usu-\husu)
\right\}^{\nicefrac12}\norm[\widetilde{a},T]{\usv}
\\
&\lesssim\left\{
\norm[a,T]{\usu-\husu}^2 + (2\mu)s_T(\husu,\husu)
\right\}^{\nicefrac12}\norm[\widetilde{a},T]{\usv}.
\end{aligned}
$$
Using again~\eqref{eq:local.norm.equiv:bis} followed by the latter inequality, we infer that
$$
\norm[a,T]{\cT\usu-\husu}
\lesssim\norm[\widetilde{a},T]{\cT\usu-\husu}
={\sup_{\usv\in\UT\setminus\{\mathsf{\vec{0}}\}}} \frac{\widetilde{a}_T(\cT\usu-\husu,\usv)}{\norm[\widetilde{a},T]{\usv}}
\lesssim\left\{
\norm[a,T]{\usu-\husu}^2 + (2\mu)s_T(\husu,\husu)
\right\}^{\nicefrac12}.
$$
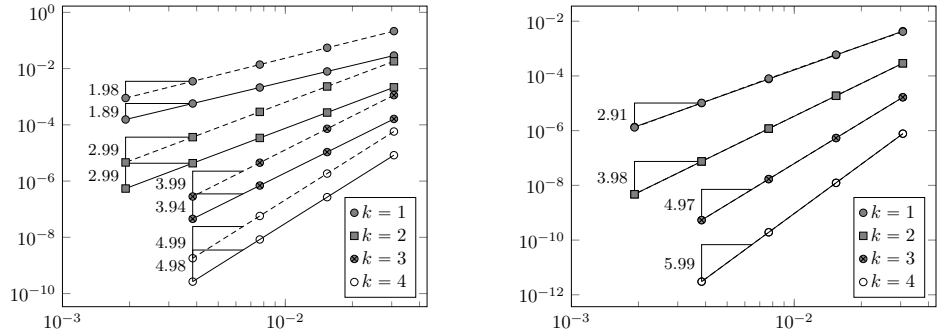
\begin{figure}
  \centering
\ifDdP
  \begin{minipage}[b]{0.4\linewidth}\centering
    \begin{tikzpicture}[scale=0.70]
      \begin{loglogaxis}[
          legend style = { legend pos = south east },
          xmin=0.001
        ]
        \addlegendimage{only marks,every mark/.append style={fill=gray},mark=*}
        \addlegendentry{$k=1$}
        \addlegendimage{only marks,every mark/.append style={fill=gray},mark=square*}
        \addlegendentry{$k=2$}
        \addlegendimage{only marks,every mark/.append style={fill=gray},mark=otimes*}
        \addlegendentry{$k=3$}
        \addlegendimage{only marks,every mark/.append style={fill=gray},mark=o}
        \addlegendentry{$k=4$}

        \pgfplotsset{cycle list name=list1}
        \addplot table[x=meshsize,y={create col/linear regression={y=err_Sh}}] {dat/ep_1_1_mesh1.dat}
        coordinate [pos=0.75] (A)
        coordinate [pos=1.00] (B);
        \xdef\slopea{\pgfplotstableregressiona}
        \draw (A) -| (B) node[pos=0.75,anchor=east] {\small\pgfmathprintnumber{\slopea}};
        \addplot table[x=meshsize,y={create col/linear regression={y=err_Sh}}] {dat/ep_1_2_mesh1.dat}        
        coordinate [pos=0.75] (A)
        coordinate [pos=1.00] (B);
        \xdef\slopeb{\pgfplotstableregressiona}
        \draw (A) -| (B) node[pos=0.75,anchor=east] {\small\pgfmathprintnumber{\slopeb}};
        \addplot table[x=meshsize,y={create col/linear regression={y=err_Sh}}] {dat/ep_1_3_mesh1.dat}
        coordinate [pos=0.75] (A)
        coordinate [pos=1.00] (B);
        \xdef\slopec{\pgfplotstableregressiona}
        \draw (A) -| (B) node[pos=0.75,anchor=east] {\small\pgfmathprintnumber{\slopec}};
        \addplot table[x=meshsize,y={create col/linear regression={y=err_Sh}}] {dat/ep_1_4_mesh1.dat}
        coordinate [pos=0.75] (A)
        coordinate [pos=1.00] (B);
        \xdef\sloped{\pgfplotstableregressiona}
        \draw (A) -| (B) node[pos=0.75,anchor=east] {\small\pgfmathprintnumber{\sloped}};

        \pgfplotsset{cycle list name=list2}
        \addplot table[x=meshsize,y={create col/linear regression={y=err_en_tuh}}] {dat/ep_1_1_mesh1.dat}
        coordinate [pos=0.75] (A)
        coordinate [pos=1.00] (B);
        \xdef\slopee{\pgfplotstableregressiona}
        \draw (A) -| (B) node[pos=0.75,anchor=east] {\small\pgfmathprintnumber{\slopee}};
        \addplot table[x=meshsize,y={create col/linear regression={y=err_en_tuh}}] {dat/ep_1_2_mesh1.dat}        
        coordinate [pos=0.75] (A)
        coordinate [pos=1.00] (B);
        \xdef\slopef{\pgfplotstableregressiona}
        \draw (A) -| (B) node[pos=0.75,anchor=east] {\small\pgfmathprintnumber{\slopef}};
        \addplot table[x=meshsize,y={create col/linear regression={y=err_en_tuh}}] {dat/ep_1_3_mesh1.dat}
        coordinate [pos=0.75] (A)
        coordinate [pos=1.00] (B);
        \xdef\slopeg{\pgfplotstableregressiona}
        \draw (A) -| (B) node[pos=0.75,anchor=east] {\small\pgfmathprintnumber{\slopeg}};
        \addplot table[x=meshsize,y={create col/linear regression={y=err_en_tuh}}] {dat/ep_1_4_mesh1.dat}
        coordinate [pos=0.75] (A)
        coordinate [pos=1.00] (B);
        \xdef\slopeh{\pgfplotstableregressiona}
        \draw (A) -| (B) node[pos=0.75,anchor=east] {\small\pgfmathprintnumber{\slopeh}};
      \end{loglogaxis}
    \end{tikzpicture}
  \end{minipage}
  \begin{minipage}[b]{0.4\linewidth}\centering
    \begin{tikzpicture}[scale=0.70]
      \begin{loglogaxis}[
          legend style = { legend pos = south east },
          xmin=0.001
        ]
        \addlegendimage{only marks,every mark/.append style={fill=gray},mark=*}
        \addlegendentry{$k=1$}
        \addlegendimage{only marks,every mark/.append style={fill=gray},mark=square*}
        \addlegendentry{$k=2$}
        \addlegendimage{only marks,every mark/.append style={fill=gray},mark=otimes*}
        \addlegendentry{$k=3$}
        \addlegendimage{only marks,every mark/.append style={fill=gray},mark=o}
        \addlegendentry{$k=4$}

        \pgfplotsset{cycle list name=list1}
        \addplot table[x=meshsize,y={create col/linear regression={y=err_uh}}] {dat/ep_1_1_mesh1.dat}
        coordinate [pos=0.75] (A)
        coordinate [pos=1.00] (B);
        \xdef\slopea{\pgfplotstableregressiona}
        \draw (A) -| (B) node[pos=0.75,anchor=east] {\small\pgfmathprintnumber{\slopea}};
        \addplot table[x=meshsize,y={create col/linear regression={y=err_uh}}] {dat/ep_1_2_mesh1.dat}        
        coordinate [pos=0.75] (A)
        coordinate [pos=1.00] (B);
        \xdef\slopeb{\pgfplotstableregressiona}
        \draw (A) -| (B) node[pos=0.75,anchor=east] {\small\pgfmathprintnumber{\slopeb}};
        \addplot table[x=meshsize,y={create col/linear regression={y=err_uh}}] {dat/ep_1_3_mesh1.dat}
        coordinate [pos=0.75] (A)
        coordinate [pos=1.00] (B);
        \xdef\slopec{\pgfplotstableregressiona}
        \draw (A) -| (B) node[pos=0.75,anchor=east] {\small\pgfmathprintnumber{\slopec}};
        \addplot table[x=meshsize,y={create col/linear regression={y=err_uh}}] {dat/ep_1_4_mesh1.dat}
        coordinate [pos=0.75] (A)
        coordinate [pos=1.00] (B);
        \xdef\sloped{\pgfplotstableregressiona}
        \draw (A) -| (B) node[pos=0.75,anchor=east] {\small\pgfmathprintnumber{\sloped}};

        \pgfplotsset{cycle list name=list2}
        \addplot table[x=meshsize,y=err_tuh] {dat/ep_1_1_mesh1.dat};
        \addplot table[x=meshsize,y=err_tuh] {dat/ep_1_2_mesh1.dat};       
        \addplot table[x=meshsize,y=err_tuh] {dat/ep_1_3_mesh1.dat};
        \addplot table[x=meshsize,y=err_tuh] {dat/ep_1_4_mesh1.dat};
      \end{loglogaxis}
    \end{tikzpicture}
  \end{minipage}
\fi
\caption{Convergence results in the energy-norm (left) and $L^2$-norm (right) for the solution to~\eqref{eq:hho} (solid lines) and its post-processing based on $\cT$ (dashed lines). The right panel shows that the post-processing has no sizable effect on element unknowns.\label{fig:convergence}}
\end{figure}
The estimate~\eqref{eq:conv.rate.cT} then follows squaring the above inequality, summing over $T\in\Th$, and using~\eqref{eq:conv.rate.en} and~\eqref{eq:approx.pT.sT}, respectively, to bound the terms in the right-hand side.
\end{proof}
To assess the estimate~\eqref{eq:conv.rate.cT}, we have numerically solved the pure displacement problem with exact solution $\vu=\big(\sin(\pi x_1)\sin(\pi x_2)+\nicefrac12x_1, \cos(\pi x_1)\cos(\pi x_2)+\nicefrac12 x_2\big)$ for $\mu=\lambda=1$ on a $h$-refined sequence of triangular meshes.
The corresponding convergence results are presented in Figure~\ref{fig:convergence}.
In the left panel, we compare the quantities on the left-hand side of estimates~\eqref{eq:conv.rate.en} and~\eqref{eq:conv.rate.cT}. Although the order of convergence is the same, the original solution $\usu[h]$ displays better accuracy in the energy-norm.
This is essentially due to face unknowns, as confirmed in the right panel, where the square roots of the quantities $\sum_{T\in\Th}\norm[T]{\su_T-\hsu_T}^2$ and $\sum_{T\in\Th}\norm[T]{\cT\usu-\hsu_T}^2$ (both of which are discrete $L^2$-norms of the error) are plotted.

\bibliographystyle{elsarticle-num}
\bibliography{cons}

\end{document}